\documentclass[12pt]{article}

\pdfoutput=1

\usepackage{amssymb,amsmath,amsthm,graphicx}
\numberwithin{equation}{section}
\newtheorem{thm}{Theorem}[section]

\newtheorem{conj}[thm]{Conjecture}

\newtheorem*{remark}{Remark}
\newtheorem*{definition}{Definition}

\newcommand{\ol}{\overline}

\def \CC{\mathbb{C}}
\def \RR{\mathbb{R}}
\def \DD{\mathbb{D}}
\def \p{\partial}
\def \e{\varepsilon}
\def\disp{\displaystyle}

\begin{document}
\author{Alexandre Eremenko\thanks{Supported by NSF grant DMS-1067886. } \hspace{.01 in} and Erik Lundberg}
\title{Non-algebraic quadrature domains}
\maketitle
\begin{abstract}
It is well known that, \emph{in the plane}, the boundary of any quadrature domain (in the classical sense) coincides with the zero set of a polynomial.
We show, by explicitly constructing some four-dimensional examples, that this is not always the case.
This confirms, in dimension $4$, a conjecture of the second author.
Our method is based on the Schwarz potential and involves elliptic integrals of the third kind.

\end{abstract}

\section{Introduction}

A bounded domain $\Omega \subset \CC$ is a \emph{quadrature domain} (in the classical sense) if it admits a formula expressing the area integral of 
any function $f$ analytic and integrable in $\Omega$ as a finite sum of weighted point evaluations of the function and its derivatives.
i.e. 
\begin{equation}\label{eq:QF1}
 \int_{\Omega} {f dA} = \sum_{m=1}^{N} \sum_{k=0}^{n_k}{a_{m,k}f^{(k)}(z_m)},
\end{equation}
where $z_m$ are distinct points in $\Omega$ and $a_{mk}$ are constants (possibly complex) independent of $f$.

{\bf Note:} This can be generalized in various directions.  One may use a different class of test functions, 
integrate with respect to a weighted density (or over the boundary), 
or allow for more general distributions in the right hand side.
``Quadrature domain in the classical sense'' is used to specify the restricted case we consider throughout this paper.

Suppose that $\Omega$ is a bounded simply-connected domain in the plane with non-singular analytic boundary.
Then the following are equivalent.
Moreover, there are simple formulas relating the details of each. \\

\noindent (i) $\Omega$ is a quadrature domain. \\
(ii) The Schwarz function of $\partial \Omega$ is meromorphic in $\Omega$. \\
(iii)  The conformal map from the disk to $\Omega$ is rational.\\

For their equivalence see \cite[Ch. 14]{Davis}.

In higher dimensions, quadrature domains are defined by replacing analytic functions with harmonic functions.
We write the quadrature formula using multi-index notation with $\alpha = (\alpha_1,\alpha_2,..,\alpha_n)$
consisting of nonnegative integers, $|\alpha| = \alpha_1 + \alpha_2+..+\alpha_n$, 
and $\partial_{\alpha} u = {\displaystyle \frac{ \partial^{|\alpha|} u}{ \partial_{x_1}^{\alpha_1}\partial_{x_2}^{\alpha_2}..\partial_{x_n}^{\alpha_n} }}$.
Then, $\Omega$ is a quadrature domain if it admits a quadrature formula for integration of harmonic functions $u$,
\begin{equation}\label{eq:QF2}
 \int_{\Omega} {u dV} = \sum_{m=1}^{N} \sum_{|\alpha|=0}^{n_k}{a_{m,\alpha} {\partial_{\alpha} u} ({\bf x_m})},
\end{equation}
where ${\bf x}_1, {\bf x}_2, .., {\bf x}_N$ are points in $\Omega$ and $a_{m,\alpha}$ are now \emph{real} constants.

{\it Remark:} In $\RR^2$, any quadrature domain for harmonic functions is a quadrature domain for analytic functions,
but not vice-versa.  
See \cite[Example 1]{Gust96} for an example of a quadrature domain for analytic functions that is \emph{not} a quadrature domain for harmonic functions.

For the case of $n \geq 3$ dimensions, condition (ii) has a counterpart formulated in terms of the 
\emph{Schwarz potential}, introduced by D. Khavinson and H. S. Shapiro (see Section~\ref{sec:SP}).
As a consequence of Liouville's Theorem on the rigidity of conformal maps \cite{DNF86}, condition (iii) does not extend to higher dimensions.
Throughout this paper, we will make use of the reformulation of quadrature domains in terms of the Schwarz potential discussed in Section \ref{sec:SP}.

Regarding the existence, in the case when there are no derivatives appearing in the quadrature formula (\ref{eq:QF2}), 
the free-boundary problem of obtaining a quadrature domain satisfying the prescribed quadrature formula can be reformulated as a so-called obstacle problem,
and the existence of a solution is proved using variational inequalities \cite{Gust81}.
The case when the quadrature formula (\ref{eq:QF2}) involves derivatives but is supported at only one point is 
especially relevant to the present paper and is discussed in Section 3.
In each of these cases, the existence theorems are true in any number of dimensions.

In the plane, quadrature domains can be constructed explicitly.
The only explicit examples in higher dimensions are a sphere in $\RR^n$ and some special examples in $\mathbb{R}^4$ discovered by L. Karp \cite{Karp92}.

The current investigation is motivated by the lack of explicit examples in the higher-dimensional case and the lack of qualitative understanding.
We address the question \emph{Are quadrature domains algebraic in $\RR^n$, with $n \geq 3$?} which was
raised by H. S. Shapiro \cite{Shapiro89} (p. 40) and posed again as an open problem by B. Gustafsson \cite{EGKP2005} (p. xv).

As usual, by ``algebraic domain'' we mean that the boundary of the domain is contained in the zero-set of a polynomial.

In the plane, under no regularity assumptions on $\p \Omega$, it was shown by D. Aharanov and H. S. Shapiro (1976) that
quadrature domains are always algebraic.  
B. Gustafsson (1983) showed that they have further nice properties \cite{Gust83}
which we summarize in the following.

\begin{thm}[B. Gustafsson, 1983]\label{thm:Gust}
If $\Omega \subset \RR^2$ is a quadrature domain, then $\Omega$ is algebraic.
Moreover, for some polynomial $P(x,y)$, 
the boundary $\partial \Omega$ consists of {\bf all} of the points of $\{P(x,y)=0\}$ except finitely many.
Moreover, the leading order (homogeneous) term in $P$ is a constant times $(X^2+Y^2)^M$ for some $M$.
Moreover, there are some explicit relations between the coefficients of $P$ and the coefficients $a_{m,k}$ in the quadrature identity.
\end{thm}

By the remark above, this applies as well to quadrature domains for harmonic functions in the plane,
so it makes sense to ask if any part of the theorem extends to $\RR^n$.
Considering that the examples in $\RR^4$ constructed by L. Karp \cite{Karp92} were algebraic,
one might hope that the answer is ``yes''.
We show that in fact quadrature domains are not always algebraic, 
thus confirming the following Conjecture \cite{Lund11} in $n=4$ dimensions.

\begin{conj}\label{conj:quad}
    In all dimensions greater than two, there exist quadrature domains that are not algebraic.
\end{conj}

Our answer is based on constructing explicit examples in Section \ref{sec:QD}.
In Section \ref{sec:LG} we use these same examples to generate some exact solutions to the Laplacian growth problem.

First, we review the definition of the Schwarz Potential in the next section.

\section{The Schwarz Potential}\label{sec:SP}

Suppose that $\Gamma$ is a non-singular, real-analytic curve in the plane.  
Then the Schwarz function $S(\zeta)$ is the function that is complex-analytic in a neighborhood of $\Gamma$ and coincides with $\bar{\zeta}$ on $\Gamma$ (see \cite{Davis} 
for a full exposition).  
If $\Gamma$ is given algebraically as the zero set of a polynomial $P(x,y)$, by the implicit function Theorem
we can obtain $S(\zeta)$ by making the complex-linear change of variables $\zeta=x+iy$, $\bar{\zeta}=x-iy$, 
and then solving for $\bar{\zeta}$ in the equation ${\displaystyle P \left( \frac{\zeta+\bar{\zeta}}{2},\frac{\zeta-\bar{\zeta}}{2i} \right)=0}$ (See \cite[p. 3]{Shapiro89}).
On the other hand, $S(\zeta)$ can also be written in terms of the conformal map $f$ from the unit disk to the domain, using the formula
\begin{equation}\label{formula}
 S(\zeta) = f^* \left( \frac{1}{f^{-1}(\zeta)}\right),
\end{equation}
where $f^*$ denotes the function obtained by conjugating the coefficients of $f$.

In the next section, we will utilize both such representations of $S(\zeta)$, so let us illustrate each by an example.

{\bf Example} (``C. Neumann's oval''): Let $a>0$ be a real parameter.  
Suppose that $\Gamma$ is the Jordan curve defined by 
$$(x^2+y^2)^2=a^2(x^2+y^2)+4 x^2.$$ 
See Figure (\ref{fig:Neumann}), for a plot of $\Gamma$ for different values of $a>0$.

\begin{figure}[ht]
    \begin{center}
    \includegraphics[scale=.15]{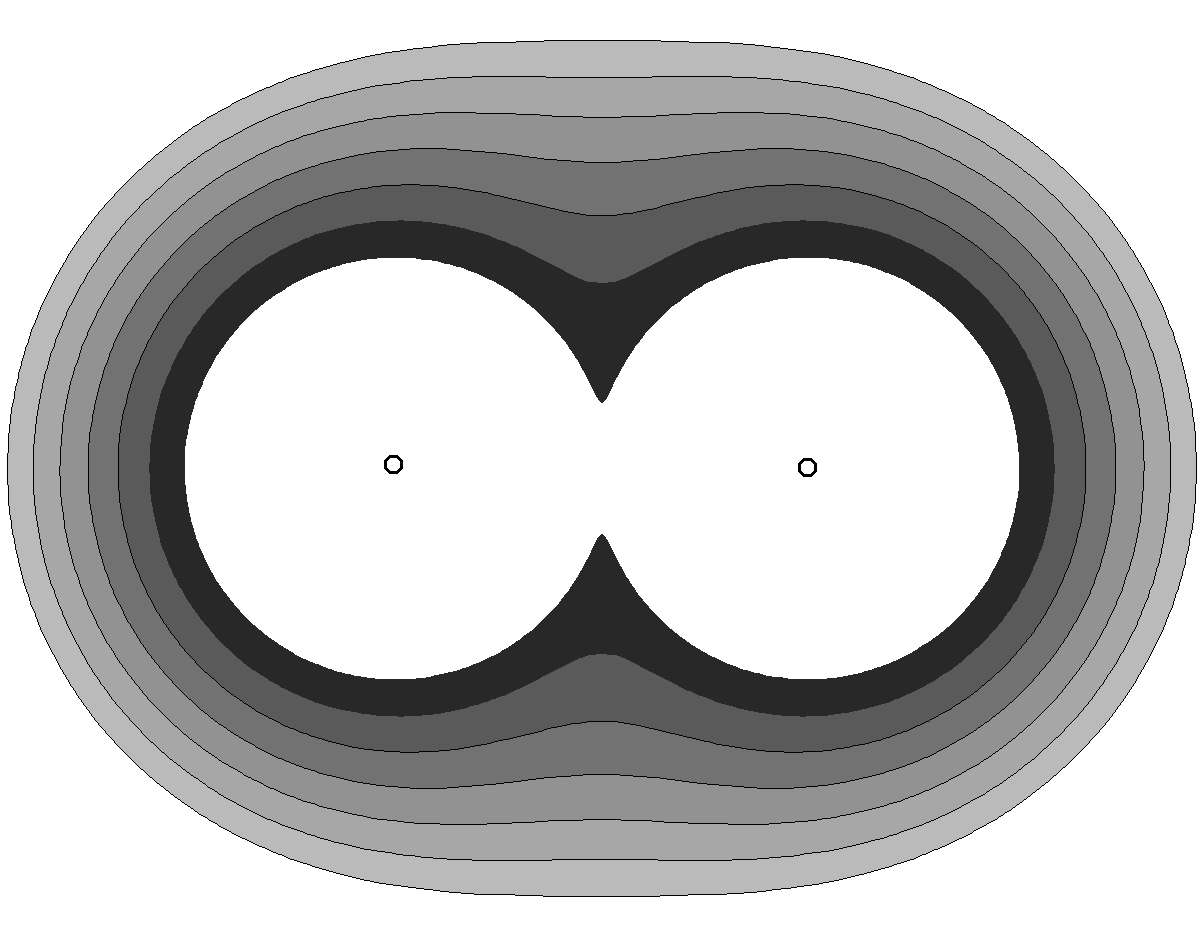}
    \end{center}
    \caption{A plot of $\Gamma$ for different values of $a>0$.}
    \label{fig:Neumann}
\end{figure}

Changing the variables we have $(\zeta \bar{\zeta})^2=a^2(\zeta \bar{\zeta})+ (\zeta+\bar{\zeta})^2$.
Solving for $\bar{\zeta}$ gives 
$$S(\zeta)=\frac{\zeta(a^2+2)+2\zeta\sqrt{a^4+a^2+\zeta^2}}{2(\zeta^2-1)},$$
indicating that $\Gamma$ is the boundary of a quadrature domain since $S(\zeta)$ is meromorphic in the interior with only poles at $\zeta = \pm 1$.

Let us make the same observation using the formula (\ref{formula}).
The conformal map from $\DD$ to the interior of $\Gamma$ is given by
$$f(z)=\frac{(R^4-1) z}{R(R^2 - z^2)},$$ 
where ${\displaystyle R = \frac{a+\sqrt{a^2+4}}{2}> 1}$.
Without actually calculating the inverse of $f$, but just observing that it is a conformal map from the interior of $\Gamma$ to $\DD$,
we can see that ${\displaystyle S(\zeta) = f^* \left( \frac{1}{f^{-1}(\zeta)}\right)}$ is 
meromorphic in the interior of $\Gamma$ and has two simple poles at $f(\pm 1/R) = \pm 1$.

Suppose that $\Gamma$ is more generally a nonsingular analytic \emph{hypersurface} in $\mathbb{R}^n$, 
and consider the following Cauchy problem posed near $\Gamma$.
The solution exists in some neighborhood of $\Gamma$ and is unique by the Cauchy-Kovalevskaya Theorem.
\begin{equation}\label{CP}
\left\{
\begin{array}{l}
\Delta w = 0 $ near $ \Gamma \\
w|_{\Gamma} =\frac{1}{2}||{\bf x}||^2 \\
\nabla w|_{\Gamma}={\bf x}
\end{array}\right.
\end{equation}

\begin{definition} 
The solution $w({\bf x})$ of Cauchy problem \ref{CP} is called the Schwarz Potential of $\Gamma$.  
\end{definition}

In $\mathbb{R}^2$, the Schwarz function can be directly recovered from the Schwarz potential.
Consider $S(\zeta) = 2\partial_\zeta w = w_x - iw_y$. 
The Cauchy-Riemann equations for $S$ follow from harmonicity of $w$, 
and $\nabla w = {\bf x}$ on $\Gamma$ implies $S(\zeta) = \bar{\zeta}$ on $\Gamma$.

As mentioned in the introduction, quadrature domains have an equivalent definition framed in terms of the Schwarz potential.
Namely, the following theorem shows that condition (ii) in the Introduction has a counterpart in higher dimensions,
where instead of ``meromorphic'', 
we have that the Schwarz potential has a single-valued real-analytic continuation throughout $\Omega$ except for finitely many points,
where $\Delta w $ is a distribution of finite order.

\begin{thm}[Khavinson, Shapiro, 1989]\label{KS}
Suppose that $\Omega$ is a connected domain in $\RR^n$ with $C^1$ boundary.  
Then $\Omega$ is a quadrature domain with quadrature formula (\ref{eq:QF2}) if and only if the Schwarz potential of $\partial \Omega$ satisfies
\begin{equation}\label{SPQF}
 w(x) = n \left( H({\bf x}) - Q({\bf x}) \right),
\end{equation}
where $H({\bf x})$ is harmonic in $\Omega$ and $C^1$ in $\overline{\Omega}$, 
and $Q({\bf x})$ is the result of applying the quadrature formula (\ref{eq:QF2}) 
to the fundamental solution $E_n({\bf x},{\bf y}) = c_n ||{\bf x} - {\bf y}||^{2-n}$ (in its second argument).
\end{thm}

{\bf Example:}  Let $\Gamma := \{{\bf x} \in \mathbb{R}^n : ||{\bf x}||^2 = r^2 \}$ be a sphere of radius $r$. 
When $n=2$, it is easy to verify that 
$$w(\zeta) = r^2 \left( \log|\zeta| + 1/2 - \log(r) \right)$$ 
solves the Cauchy Problem (\ref{CP}), and in higher dimensions the Schwarz potential is 
$$w({\bf x}) = \frac{n}{2(n-2)} r^2 - \frac{r^n}{(n-2)||{\bf x}||^{n-2}}.$$
Let us check that this agrees with Theorem \ref{KS}.
The mean-value property for the ball of radius $r$ about zero, $B_r({\bf 0})$, gives the quadrature formula:
$$\int_{B_r({\bf 0})} u dV = \text{Vol}\left(B_r\right) u(0).$$
Applying this quadrature formula to $E_n({\bf x},{\bf y})$ in its second argument gives 
$$Q({\bf x}) = \text{Vol}(B_r) E_n({\bf x}) = \text{Vol}(B_r) c_n ||{\bf x}||^{2-n} = r^n \text{Vol}(B_1) c_n ||{\bf x}||^{2-n}$$
$$ = \frac{r^n}{n(n-2)} ||{\bf x}||^{2-n},$$
since the constant $c_n$ appearing in the fundamental solution is ${\displaystyle \frac{1} {(n-2)\omega_n}}$, where $\omega_n$ is the surface area of the unit sphere.
Taking ${\displaystyle H({\bf x}) = \frac{n}{2(n-2)} r^2}$, we obtain the same formula (\ref{SPQF}).

\section{An example of a non-algebraic quadrature domain}\label{sec:QD}

A simple case of particular interest is when the quadrature formula (\ref{eq:QF2}) is supported at a single point.

\begin{equation}\label{eq:QFsingle}
 \int_{\Omega} {u dV} = \sum_{|\alpha|=0}^{M}{a_{\alpha} {\partial_{\alpha} u} (0)},
\end{equation}
where $\alpha = (\alpha_1,\alpha_2,..,\alpha_n)$ again is a multi-index.

This quadrature formula corresponds to a domain that has only finitely many non-vanishing harmonic moments.
In two dimensions, this is equivalent to the conformal map from the disk being a polynomial.

Under the condition that the coefficients $a_{\alpha}$ for $|\alpha| > 0$ are small in comparison with the leading coefficient $a_{(0,0,..,0)}$, 
the existence and uniqueness of domains 
admitting a quadrature formula of the type (\ref{eq:QFsingle}) 
follows from one of the earliest results on the solvability of the inverse potential problem \cite{Sretenskii38} (cf. \cite{Ivanov56}).

Let us further specialize to the case when all derivatives in the quadrature formula (\ref{eq:QFsingle}) 
are with respect to the same variable, say $x_1$.

\begin{equation}\label{eq:QF3}
 \int_{\Omega} {u dV} = \sum_{k=0}^{m} a_k \frac{\partial^k u}{\partial x_1 ^k}(0).
\end{equation}

In \cite{Karp92}, special cases were explicitly generated by rotating the limacon 
$D:=\{\zeta: \zeta = w + \sigma w^2, |w| < 1\}$ about the $x$-axis, where $\sigma > 1/2$ is a real parameter.
Before rotation, this starts out as a two-dimensional quadrature domain with quadrature formula
$$\int_{D} {f(\zeta) dA} = \pi[(1+2\sigma^2) f(0) + \sigma  f'(0)].$$

The four-dimensional domain generated by rotation is a quadrature domain with quadrature formula
$$\int_{\Omega} {u dV} = \pi^2 \left[\frac{(1+6\sigma^2 + 2 \sigma^4)}{2} u(0) + 
\frac{\sigma(1+2\sigma^2)}{2} \frac{\partial u}{\partial x_1}(0) + \frac{\sigma^2}{12}\frac{\partial^2 u}{\partial x_1 ^2}(0) \right].$$
where $x_1$ corresponds to the axis of symmetry.

We prove Conjecture \ref{conj:quad} in $\RR^4$ by constructing a domain admitting the quadrature formula (\ref{eq:QF3}) 
with $m=1$, $a_0 > 0$, and $a_1 > 0$.

\begin{thm}\label{thm:nonalg}
In $\RR^4$, there exist quadrature domains that are not algebraic,
namely, admitting the quadrature formula (\ref{eq:QF3}) with $m=1$, $a_0>0$, and $a_1>0$ (with $a_1$ small in comparison with $a_0$).
Moreover, the boundary can be described explicitly in terms of elliptic integrals of the third kind.
\end{thm}

We will work in terms of the Schwarz potential introduced in the previous section.
Before beginning the proof we make an observation regarding axially symmetric potentials.

Recall the axially symmetric reduction of Laplace's equation:

Suppose that $u(x_1,x_2,...,x_n)$ is harmonic in $\RR^n$ and axially symmetric about the $x_1$-axis.
Write $U(x,y) = u(x_1,x_2,...,x_n)$, where $x=x_1$ and $y = \sqrt{x_2^2+...+x_n^2}$.
Then, 
$$\Delta U + \frac{(n-2)U_y}{y} = 0.$$
Indeed, $\Delta u = \text{div} \left( U_x,U_y y_{x_2},U_y y_{x_3},...,U_y y_{x_n}\right)$

\hspace{.6 in} ${\disp = U_{xx} + U_{yy} \left(\sum_{i=2}^n y_{x_i}^2 \right) + U_y \left(\sum_{i=2}^n y_{x_i x_i} \right),}$

where easy calculations give $\sum_{i=2}^n y_{x_i}^2 = 1$ and ${\displaystyle \sum_{i=2}^n y_{x_i x_i} = \frac{(n-2)}{y}}$.

In the case $n=4$, $U$ satisfies the equation ${\displaystyle \Delta U + \frac{2 U_y}{y} = 0}$,
iff $y U(x,y)$ is a harmonic function of two variables.
Indeed, 
$$\Delta (y U) = y \Delta U + 2 \nabla U \cdot \nabla y + U \Delta y = y \Delta U + 2 U_y.$$
This fact about axially symmetric potentials in $\mathbb{R}^4$ was used in \cite{Khav91} and \cite{Karp92} where it is explained in more detail.

If $\Omega$ is a domain in $\RR^4$ with axially symmetry, then the above considerations apply to the Schwarz potential, 
since the rotational-symmetry of the Cauchy data (\ref{CP}) is passed to the solution $w$.

Thus, the Cauchy problem (\ref{CP}) defining the Schwarz potential $w$ is reduced to the following two-dimensional Cauchy problem
for $W(x,y) = w(x_1,x_2,x_3,x_4)$:
\begin{equation}\label{CPW}
\left\{
\begin{array}{l}
{\displaystyle \Delta W +  \frac{2 W_y}{y}} = 0 $, near $ \Gamma \\
{\displaystyle W|_{\gamma} =  \frac{1}{2} (x^2+y^2) } \\
\nabla W|_{\gamma}=\langle x,y \rangle
\end{array}\right.
\end{equation}
where $\gamma$ is the symmetric curve in the plane whose rotation generates the boundary of $\Omega$.

According to the previous observation, we have
\begin{equation}\label{eq:harmonic}
 \Delta \left( y W(x,y) \right) = 0,
\end{equation}
which will be used below.

\begin{proof}[Proof of Theorem \ref{thm:nonalg}]

We will construct $\Omega$ by first describing a conformal map $f:\DD \rightarrow D_p$ from the unit disk $\DD$ 
to a domain $D_p$ in the plane symmetric with respect to the real axis.
Then we will take $\Omega$ to be the domain generated by rotation of $D_p$ into $\RR^4$ about the $x$-axis.
i.e., $(x_1,x_2,x_3,x_4) \in \Omega$ if and only if $\left( x_1,\sqrt{x_2^2 + x_3^2 + x_4^2} \right) \in D_p$.
Notice that the boundary $\Gamma = \p \Omega$ is algebraic if and only if $\gamma = \p D_p$ 
is algebraic, in which case $\gamma = \{ \rho(x,y) = 0 \}$, and $\Gamma = \left\{ \rho \left( x_1,\sqrt{x_2^2 + x_3^2 + x_4^2} \right) = 0 \right\}$,
where $\rho(x,y)$ is a polynomial that is even in the variable $y$.

The conformal map $f$ and its relevant properties will be established in the proofs of the following claims.

\noindent {\bf Claim 1:} There exists a function $f$ with real coefficients, analytic and univalent in a neighborhood of $\overline{\DD}$, with $f(0) = 0$, 
so that the function
\begin{equation}\label{eq:fct1}
\left(f(z) - f\left(\frac{1}{z}\right)\right)^2 = g(z)
\end{equation} 
is rational, and its only pole in $\ol{\DD}$ is a pole of exact order $3$ at $z=0$.

\noindent {\bf Claim 2:} The (analytically continued) function $f$ has an infinitely-sheeted Riemann surface.

We defer the proofs of the claims in favor of first seeing how they are used to prove the theorem.

Since, by Claim 1, $D = f(\DD)$ is the image of $\DD$ under a univalent function analytic in a neighborhood of $\overline{\DD}$,
the boundary of $D_p$, and therefore of $\Omega$, is analytic and non-singular.  
Thus $\partial \Omega$ has a Schwarz potential $w(x_1,x_2,x_3,x_4)$, which by axial symmetry can be reduced to a function $W(x,y)$ of two variables
(see the discussion above just before the proof).
As stated in (\ref{eq:harmonic}), $yW(x,y)$ is harmonic near $\partial D_p$.

Let $\zeta = x + iy$. 
Then the function 
$$V(x,y) = y W(x,y) - \Im \{ \zeta^3 \} / 12,$$ 
is harmonic 
(subtracting the harmonic term $\Im \{ \zeta^3 \} / 12$ simplifies the following calculations without changing the singularities of $V$).

Next, consider the (complex-analytic) function obtained by taking the $\zeta$-derivative $V_\zeta = \frac{1}{2}(V_x - i V_y)$.
$$V_\zeta = \frac{-i}{2}W + y W_\zeta + \frac{i \zeta^2}{4}$$

The Cauchy data for $W$, stated in (\ref{CPW}), can be written $W = \zeta \bar{\zeta}/2$ and $W_\zeta = \bar{\zeta}/2$,
so that on the boundary $\gamma = \partial D_p$, the equation above gives
$$V_\zeta|_{\gamma} = \frac{i \zeta \bar{\zeta}}{4} + \frac{\zeta-\bar{\zeta}}{4i}(\bar{\zeta}) + \frac{i\zeta^2}{4} = \frac{i}{4} \left( \zeta - \bar{\zeta} \right)^2,$$
where we have also used ${\displaystyle y =  \frac{\zeta - \bar{\zeta}}{2i}}$.

Now we replace $\bar{\zeta}$ with $S(\zeta)$ the Schwarz function of $\gamma$.
Then, both sides of the equation are analytic in a neighborhood of $\gamma$, and thus the following becomes an identity (not just valid on the boundary).
\begin{equation}\label{eq:Vzeta}
 V_\zeta = \frac{i}{4} \left( \zeta - S(\zeta) \right)^2.
\end{equation}

This relates the $4$-dimensional Schwarz potential to the $2$-dimensional Schwarz function.

As follows from (\ref{formula}) when the coefficients of $f$ are real, 
the Schwarz function of $\gamma$ satisfies the functional equation
$$S(f(z)) = f \left( \frac{1}{z} \right).$$
Using this relationship, and substituting $\zeta = f(z)$ into (\ref{eq:Vzeta}), we get
$$V_{\zeta}(f(z)) = \frac{i}{4}{\displaystyle \left(f(z) - f\left(\frac{1}{z}\right)\right)^2}.$$
By Claim 1 the right hand side is $g(z)$, a function analytic in a neighborhood of $\ol{\DD}$, except for a pole at $z=0$ of order $3$.
So, ${\displaystyle g(z) = \frac{Q(z)}{z^3} }$, with $Q(z)$ analytic in a neighborhood of $\ol{\DD}$ and $Q(0) \neq 0$.

Since $f(0) = 0$ and $f$ is univalent, we have ${\displaystyle f^{-1}(\zeta) = \zeta h(\zeta) }$, for some $h(\zeta)$ analytic and non-vanishing in $D_p$.
Thus,
$$V_{\zeta}(\zeta) = g(f^{-1}(\zeta)) = \frac{1}{\zeta^3} \frac{Q(f^{-1}(\zeta))}{h(\zeta)^3},$$
which is analytic in $D_p$ except for a pole of order $3$ at $\zeta = 0$.

This implies that (see \cite[p. 183]{Karp92}) 
$$w(x_1,x_2,x_3,x_4) = w({\bf x}) = A \cdot ||{\bf x}||^{-2} + B \cdot \p_{x_1} (||{\bf x}||^{-2}) + H({\bf x}),$$ 
where $A$ and $B$ are constants and $H({\bf x})$ is harmonic.

Since, up to a constant, $||{\bf x}||^{-2}$ is the fundamental solution $E_4({\bf x},{\bf 0})$ evaluated at ${\bf y}={\bf 0}$, 
Theorem \ref{KS} implies $\Omega$ is a quadrature domain with quadrature formula of the form
$$\int_{\Omega} {u dV} = a_0 u({\bf 0}) + a_1 \frac{\partial u}{\partial x_1}({\bf 0}).$$

Next we apply Claim 2 in order to show that the boundary of $\Omega$ is not algebraic.

{\bf Note:} It does not follow immediately from Claim 2 that the image of $\partial \DD$ is non-algebraic, 
since it is possible for an algebraic domain to be conformally mapped from the unit disk by a transcendental function.
Indeed, any non-circular ellipse is such an example.

Suppose, to the contrary, that $\gamma = \p D_p$ is the zero set of a polynomial.
Then, as observed in Section \ref{sec:SP}, the Schwarz function $S(\zeta)$ is an algebraic function, 
since it can be obtained by solving for $\bar{\zeta}$ in the equation 
$$P \left( \frac{\zeta+\bar{\zeta}}{2},\frac{\zeta-\bar{\zeta}}{2i} \right)=0.$$

Substitute $z = f^{-1}(\zeta)$ into (\ref{eq:fct1}) and use the formula (\ref{formula}) 
$$ S(\zeta) = f \left( \frac{1}{f^{-1}(\zeta)}\right),$$
where again, the conjugation is missing because the coefficients of $f$ are real.
We obtain
$$g(f^{-1}(\zeta)) = (\zeta - S(\zeta))^2.$$
Since $g(z)$ is rational, and $S(\zeta)$ is algebraic, this implies
$$f^{-1}(\zeta) = g^{-1}\left( (\zeta - S(\zeta))^2 \right)$$ 
is also an algebraic function.
This contradicts the fact that $f^{-1}(\zeta)$ is the inverse of a transcendental function.
We conclude that $S(\zeta)$ is transcendental, and the boundary of $\Omega$ is not contained in the zero set of a polynomial.

\end{proof}

It remains to prove the claims.

\begin{proof}[Proof of Claim 1]

With $a>0, C > 0$ real parameters, let
$$ h(z) = C\frac{z^2-1}{z}\sqrt{\frac{z+a}{z}(1+az)} .$$
It is easy to see that the square root above has two single valued analytic branches on $|z| = 1$.
We choose that branch which is $>0$ at $z=1$.

We observe that
\begin{equation}\label{eq:sym}
 h(1/z) = - h(z).
\end{equation}

Since $h(z)$ is analytic in a neighborhood of $\p \DD$, say $\{1-\e < |z| < 1+\e \}$ it has a Laurent expansion
$$h(z) = \sum_{j= -\infty}^{\infty}{c_j z^j}.$$

Let $f(z) = \sum_{j=1}^{\infty}{c_j z^j}$.
By (\ref{eq:sym}), $c_j = -c_j$, and $c_0 = 0$.  

Therefore, 
\begin{equation}\label{eq:h2}
 h(z) = f(z) - f \left( \frac{1}{z} \right),
\end{equation}
where $f(z)$ is analytic in a neighborhood of $\ol{\DD}$, and $f(1/z)$ is analytic in a neighborhood of $\CC \setminus \DD$ 
with $f(1/z) \rightarrow 0$ as $z \rightarrow \infty$.

Observe that $f(z)$ satisfies (\ref{eq:fct1}) if we take
$$ g(z) = h(z)^2 = C^2\frac{(z^2-1)^2(z+a)(1+az)}{z^3}.$$
We remark that, up to changing the sign of the parameter $a$, this is the only choice of $g$ for which all the conditions in the claim can be satisfied.

For $w \in \DD$, $f(w)$ can be expressed as the contour integral 
$$f(w) = \frac{1}{2 \pi i}\int_{|z|=1}{\frac{h(z)}{z-w}dz}.$$

Indeed, by (\ref{eq:h2})
$$\int_{|z|=1}{\frac{h(z)}{z-w}dz} = \int_{|z|=1}{\frac{f(z) - f(\frac{1}{z})}{z-w}dz} = \int_{|z|=1}{\frac{f(z)}{z-w}dz} - \int_{|z|=1}{\frac{f(\frac{1}{z})}{z-w}dz},$$
and we have ${\displaystyle \int_{|z|=1}{\frac{f(1/z)}{z-w}dz} = 0}$, 
since the integrand is analytic in $\CC \setminus \overline{\DD}$ and is $O(1/|z|^2)$ as $z \to \infty$.
Moreover, since $f(z)$ is analytic in a neighborhood of $\overline{\DD}$, the Cauchy integral formula gives ${\displaystyle \frac{1}{2\pi i}\int_{|z|=1}{\frac{f(z)}{z-w}dz} = f(w)}$.

So we have
\begin{equation}\label{eq:CauchyInt}
 f(w) = \frac{C}{2\pi i}\int_{|z|=1}{\frac{(z^2-1)\sqrt{(1+\tfrac{a}{z})(1+az)}}{z(z-w)}dz}.
\end{equation}

Next, we show that $f(w)$ is univalent at least for sufficiently small values of the parameter $0<a<1$.
The parameter $C$ does not affect the univalence of $f(w)$, so we fix $C=1$.

Use (\ref{eq:CauchyInt}) to write the derivative $f'(w)$:
$$f'(w) = \frac{-1}{2\pi i}\int_{|z|=1+\varepsilon}{\frac{(z^2-1)\sqrt{(1+\tfrac{a}{z})(1+az)}}{z(z-w)^2}dz}, $$
where we have also deformed the contour slightly with $\varepsilon > 0$.
As $a \rightarrow 0$, the integrand converges (uniformly for $|z|=1+\varepsilon$) to ${\displaystyle \frac{(z^2-1)}{z(z-w)^2} }$.
Thus, $f'(w)$ converges uniformly in $\ol{\DD}$ to
$$\frac{1}{2\pi i}\int_{|z|=1+\varepsilon}{\frac{z}{(z-w)^2} - \frac{1}{z(z-w)^2} dz} = 1,$$
using residues.

This guarantees univalence of $f(w)$ throughout $\ol{\DD}$ for all sufficiently small $a>0$.
Indeed, for any two points $w_1$ and $w_2$ in $\ol{\DD}$, 
$$f(w_2) - f(w_1) = \int_{w_1}^{w_2}{f'(w) dw},$$
and $f'(w)$ is uniformly close to $1$ for $a>0$ sufficiently small.

\end{proof}

\begin{proof}[Proof of Claim 2]

We investigate the global analytic function whose principal branch is given by the integral (\ref{eq:CauchyInt}).

We will deform the contour in (\ref{eq:CauchyInt}), 
but first we manipulate the integrand so that the new integral will converge.
Using
$$\frac{(z^2-1)}{z(z-w)} = 1 + \frac{w}{z} + \frac{w^2-1}{z(z-w)},$$
we have 
$$f(w) = A_0 + A_1 w + C(w^2-1) F(w), $$
where $A_0$ and $A_1$ are integrals that do not depend on $w$, and 
\begin{equation}
 F(w) = \frac{1}{2\pi i}\int_{|z|=1}{\frac{\sqrt{(1+\tfrac{a}{z})(1+az)}}{z(z-w)}dz}.
\end{equation}
Now deform the contour until it it ``just surrounds'' the segment $(-\infty,-1/a]$.  
The integrand changes sign as it switches sides, and the path changes direction. 
Thus, we obtain twice the integration along a segment:
\begin{equation}\label{eq:mainF}
 F(w) = \frac{1}{\pi i}\int_{-\infty}^{-1/a}{\frac{\sqrt{(1+\tfrac{a}{z})(1+az)}}{z(z-w)}dz} = \frac{1}{\pi i}\int_{-\infty}^{-1/a}{\frac{\sqrt{G(z)}}{(z-w)}dz},
\end{equation}
where $G(z) = (1+\tfrac{a}{z})(\tfrac{1}{z}+a)\tfrac{1}{z}$.
By further manipulation, this can be expressed in terms of a standard form of a \emph{complete elliptic integral of the third kind},
which we will do at the end of this section after the proof.
For now, we show directly from (\ref{eq:mainF}) that the analytic continuation of $F(w)$, and hence $f(w)$, has an infinitely-sheeted Riemann surface.

The monodromy for $F(w)$ can be obtained as a nice exercise in applying the Sokhotski-Plemelj relation \cite[Ch. 14]{Henrici86}.
Accordingly, the jump across the segment $(-\infty,-1/a]$ of such an integral as appears in (\ref{eq:mainF})
is equal to $2 \pi i$ times the numerator of the integrand.
i.e., let $w$ approach a point $x$ on the segment $(-\infty,-1/a]$ from above and below, respectively, then
\begin{equation}\label{PPSW}
 F(x_+) - F(x_-) = -2\sqrt{G(x)}.
\end{equation}

In order to perform analytic continuation, choose a base point $w_0$, say in the upper half-plane, 
and let $F_0(w)$ denote the principal branch (which is expressed by the integral itself).
Then, let $\gamma_1$ be a loop based at $w_0$ that winds once around $1/a$.
According to (\ref{PPSW}), performing analytic continuation along $\gamma_1$, we obtain a new branch $F_1(w) = F_0(w) + 2\sqrt{G(w)}$.
If $\gamma_1$ is traced again, then we return to the original branch since $2\sqrt{G(w)}$ switches sign and cancels the contribution from $F_0$.
So $w=1/a$ only has ramification index 2.

Consider a second loop $\gamma_2$ based at $w_0$ that winds around the origin, crossing the segment $[0,a]$.
Traveling along this loop does nothing to the initial branch $F_0$, but if we first continue along $\gamma_1$, then the branch $F_1$
will be affected when we continue along $\gamma_2$.
Alternate between these two loops and perform analytic continuation along $\gamma_1 \cdot \gamma_2 \cdot \gamma_1 \cdots \gamma_1$, 
where $\gamma_1$ is traced $k+1$ times.
Let $F_i(w)$ denote the new branch obtained at each step.  Then we have 
$$F_1(w) = F_0(w) + 2 \sqrt{G(z)}$$
$$F_2(w) = F_0(w) - 2 \sqrt{G(z)}$$
$$F_3(w) = F_0(w) + 2 \sqrt{G(z)} + 2 \sqrt{G(z)}$$
$$F_4(w) = F_0(w) - 4 \sqrt{G(z)}$$
$$F_5(w) = F_0(w) + 2 \sqrt{G(z)} + 4 \sqrt{G(z)}$$
$$F_{2k+1}(w) = F_0(w) + 2(k+1)\sqrt{G(z)}$$
Thus, the Riemann surface for $F(w)$ is infinitely-sheeted.

\end{proof}

{\bf Reduction of the integral to a standard form:} 
Let us make $f(w)$ more explicit by reducing $F(w)$ to a standard form.
Make the substitution $z = -\frac{1}{\xi}$, $dz = \frac{d\xi}{\xi^2}$, in (\ref{eq:mainF}) to get
$$ F(w) = \frac{1}{\pi i}\int_{0}^{a}{\frac{\sqrt{(1-a\xi)(1-a/\xi)}}{(1+\xi w)}d\xi} = \frac{1}{\pi i}\int_{0}^{a}{\frac{\sqrt{(1-a\xi)(\xi-a)\xi}}{\xi(1+\xi w)}d\xi}.$$
Using $\frac{1}{\xi(1+\xi w)} = \frac{1}{\xi} - \frac{w}{1+\xi w}$,
$$F(w) = C_0 + \frac{\sqrt{a}}{\pi } \int_{0}^{a}{\frac{\sqrt{(\xi-1/a)(\xi-a)\xi}}{(\xi + 1/w)}d\xi},$$
where $C_0$ is an integral that doesn't depend on $w$.

For the remaining integral we write 

$\int_{0}^{a}{\frac{\sqrt{(\xi-1/a)(\xi-a)\xi}}{(\xi + 1/w)}d\xi} = (\xi-1/a)(\xi-a)\xi \int_{0}^{a}{\frac{1}{(\xi + 1/w)\sqrt{(\xi-1/a)(\xi-a)\xi}}d\xi}.$

This integral on the right hand side is entry $4$ from section $3.137$ of \cite{Grad},
where it is expressed in terms of a complete elliptic integral of the third kind.
The argument $w$ appears only within the so-called \emph{elliptic characteristic}.

\section{Laplacian growth in $\RR^n$}\label{sec:LG}

Given an initial domain $\Omega_0$, consider the following moving-boundary problem.
Find a one-parameter family of domains, $\{ \Omega_t \} \subset \mathbb{R}^n$ so that the normal velocity, 
$v_n$, of the boundary $\Gamma_t := \partial \Omega_t$ is determined by Green's function, $P({\bf x},t)$, 
of $\Omega_t$ with a fixed singularity positioned at ${\bf x}_0 \in \Omega_t$.
\begin{equation}\label{LG}
\left\{
\begin{array}{l}
v_n|_{\Gamma_t}=-\nabla P \\
\Delta P = 0 $, in $ \Omega_t \\
P|_{\Gamma_t} = 0 \\
P({\bf x } \rightarrow {\bf x_0},t) \sim Q \cdot E_n({\bf x},{\bf x_0}),
\end{array}\right.
\end{equation}
where $E_n$ is the fundamental solution of the Laplace equation with singularity at ${\bf x}_0$ and $Q>0$ ($Q<0$) determines the injection (suction) rate at the source (sink) ${\bf x}_0$.

If the domains $\Omega_t$ are bounded as in the case considered below, with $Q>0$ problem (\ref{LG}) actually produces a \emph{shrinking} boundary.  
We get a growth process if $\Omega_t$ contains infinity so $P$ then solves an \emph{exterior} Dirichlet problem.  
In such a situation, it is common to place the sink at infinity by prescribing asymptotics for $\nabla P$ 
so that the flux across neighborhoods of infinity is proportional to $Q$.  

This is a \emph{non-linear} moving boundary problem that models viscous fingering in a Hele-Shaw cell (when $n=2$) or bubble growth in a porous media (when $n=3$).
These processes exhibit complicated pattern formation \cite{MPST2006},
yet miraculously, there turns out to be an abundance of explicit, exact solutions in the plane.
This miracle is partly explained by Richardson's Theorem \cite{Rich72}, which guarantees that the property of $\Omega$ being a quadrature domain is
preserved under Laplacian growth, and moreover the time-evolution of the quadrature formula is simple.

\begin{thm}[S. Richardson]
\label{Rich}
If $\Omega_t$ solves Problem (\ref{LG}), then for any harmonic function $u$
$$\frac{d}{dt} \int_{\Omega_t}{u dV} = Q u(x_0).$$
\end{thm}

Given the many equivalent definitions of quadrature domains a few of which were mentioned in the introduction, 
it is not surprising that there are alternative statements of Richardson's Theorem.  
For instance, applying the Theorem to a basis of harmonic polynomials expanded about the point $x_0$,
we obtain that all harmonic moments are constant except one (infinitely-many conservation laws).
One may ask if there is a moment-generating function for the harmonic moments.
The answer is that it is precisely the exterior gravitational potential of $\Omega_t$,
which according to Richardson's Theorem evolves by inheriting the effect due to a single additional point-mass at $x_0$ growing linearly in strength.

Yet another interpretation is relevant to the approach taken in the above sections.
Namely, all singularities of the Schwarz potential are stationary except one positioned at $x_0$ which does not move but simply grows in strength.
This is a consequence of Richardson's Theorem combined with Theorem \ref{KS}.
It can be seen more directly by establishing the following formula, shown in \cite{Lund11} from elementary calculations, relating the Schwarz potential $w$ to the pressure $P$.
\begin{equation}\label{DS}
\frac{\partial}{\partial t}w({\bf x},t) =-nP({\bf x},t),
\end{equation}
where $n$ is the spatial dimension.

In the statement of the problem (\ref{LG}), if we allow $P$ to have more exotic singularities, 
then the quadrature domains constructed in \cite{Karp92} provide exact solutions in $\RR^4$ with a dipole flow superimposed on a source of varying rate. 
Some exact solutions without multipole flows were described in \cite{Lund11}, 
but they required a combination of sources and sinks.
The examples constructed in Section \ref{sec:QD} provide exact solutions to the Problem (\ref{LG}) as stated (i.e. with a single sink).
Before describing the solution, we review its analogue in the plane, which is the well-known example (due to Polubarinova-Kochina \cite{Kochina91}) that encounters a cardioid.

{\bf Example 1 ($\RR^2$):}  Consider the family of domains $D$ with boundary given by the curves $\{\zeta:\zeta=a z^2+bz,|w|<1 \}$ with $a>0$, $b>2a$ real.  
The Schwarz function is given by $S(\zeta)= -2ab/(a-\sqrt{a^2+4b\zeta})+4b^3/(a-\sqrt{a^2+4b\zeta})^2$ which has a single-valued branch in the interior of the curve 
for appropriate parameter values $a$ and $b$.  
The only singularities of the Schwarz function interior to the curve are a simple pole and a pole of order two at the origin.
Since $S(\zeta) = 2 \partial_\zeta w(x,y,t)$, Equation (\ref{DS}) becomes $\frac{\partial}{\partial t}S(\zeta,t) = -4 \partial_\zeta P(x,y,t)$.
So, in order for a one-parameter family of domains to solve Problem (\ref{LG}), the singularities of $S(\zeta)$ must be time-independent except for one simple pole.
Given an initial domain from this family we can choose a one-parameter slice of domains so that the simple pole increases (resp. decreases) 
while the pole of order two does not change.  
This gives an exact solution to the Laplacian growth problem with injection (resp. suction) taking place at the origin.  
In the case of injection, the domain approaches a circle.  
In the case of suction, the domain develops a cusp in finite time.

{\bf Example 2 ($\RR^4$):} Similarly, 
the the domains constructed in the previous section contain one-parameter families 
that solve the Laplacian growth problem with a sink at the origin.
Since changing $C$ simply scales the domain, the different shapes that can occur in these one-parameter families are determined by the value of $a$.
Near the value $a \approx 0.82217...$, the boundary develops a cusp (apparent in Fig. \ref{fig:R4}).

\begin{figure}[ht]
    \begin{center}
    \includegraphics[scale=.2]{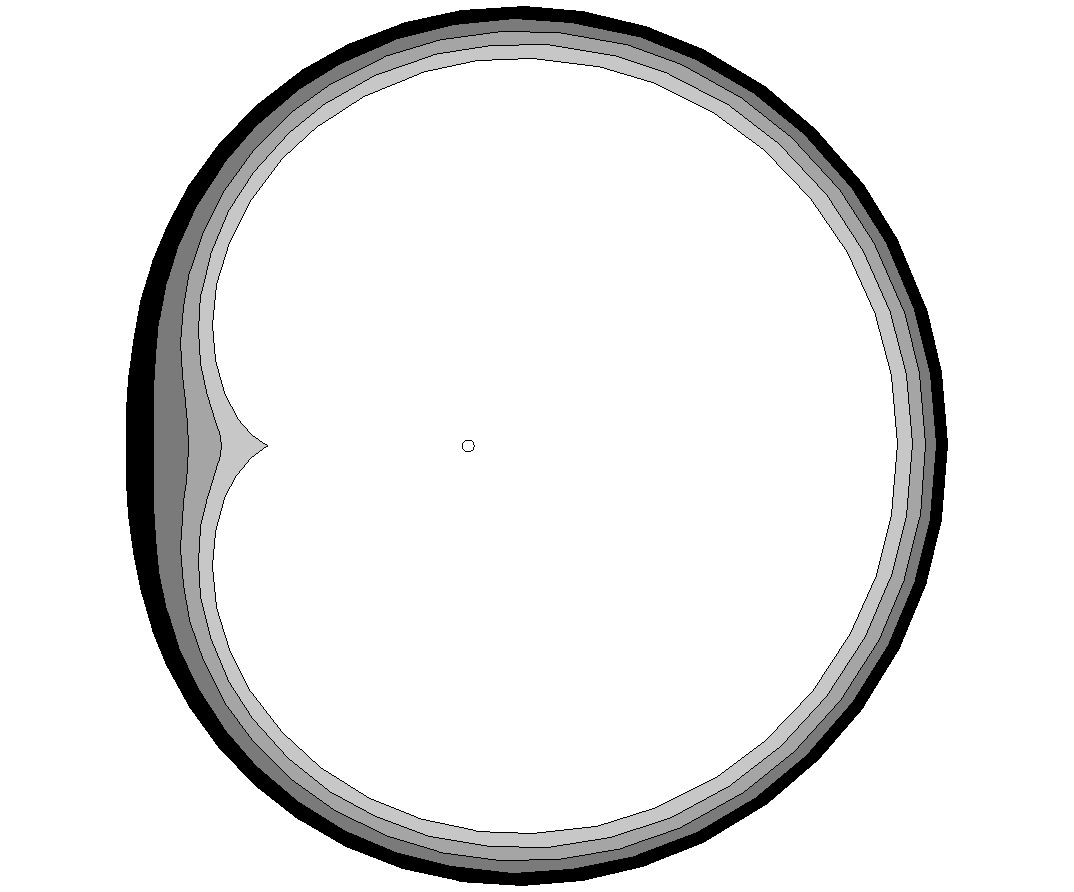}
    \end{center}
    \caption{Two-dimensional profile of a solution of Laplacian growth in $\RR^4$, reminiscent of the cardioid example in the plane.}
    \label{fig:R4}
\end{figure}

\section{Concluding remarks}

{\bf 1.} Before rotation into $\RR^4$, the domain $D_p$ is not a quadrature domain, but it is a quadrature domain \emph{in the wide sense} 
and admits a formula for the integral of any function $F$ analytic in $D_p$
in terms of a distribution supported on the interval $[-\alpha,0]$,
where $-\alpha = f(-a)$ is the image of $-a$ under the conformal map described in the proof of Theorem \ref{thm:nonalg}.

{\bf 2.} It is possible to obtain an explicit (but complicated) formula for the Schwarz potential of an axially symmetric domain in $\RR^n$.
First, take the axially symmetric reduction to (two variables) as discussed just before the proof of Theorem \ref{thm:nonalg}:
$$\Delta U + \frac{(n-2)U_x}{x} = 0,$$
where $U(x,y)$ is the Schwarz potential in the cylindrical variables, 
and we have taken $y$ as the variable along the axis of symmetry in order to correspond with a reference used below.

Now, let $x$ and $y$ each take complex values and make the change of variables to characteristic coordinates $z=x+iy$, $w=x-iy$:
$$U_{z w} + \frac{(n-2)(U_z + U_w)}{2(z+w)} = 0.$$

This is a complexified hyperbolic equation with lower order terms.
The Cauchy problem can be solved using a generalization of d'Alembert's formula involving a so-called Riemann function $A(s,t;z,w).$
The general procedure is discussed in \cite[Ch. 4 and Ch. 5]{Gar}, and the Riemann function for this equation is given by formula (5.36).
$$A(s,t;z,w) = \frac{(s+t)^{\lambda}}{(s+w)^{\lambda/2}(t+z)^{\lambda/2}}F \left(\tfrac{\lambda}{2},\tfrac{\lambda}{2};1;\frac{(s-z)(t-w)}{(s+w)(t+z)}\right),$$
where $F(a,b;c;x)$ denotes the Gauss hypergeometric function, and $\lambda = (n-2)$.
In the case $n = 4$, the hypergeometric function $F(1,1;1;x) = \frac{1}{1-x}$ is rational, but in the case $n=3$ it is transcendental.

Let $S(z)$ be the Schwarz function of the curve $\gamma$ that generates the surface (by rotation).
The complexification of the curve then becomes the graph in $\CC^2$ of $S(z)$.
Using also the fact that the inverse $S^{-1}(w) = \bar{S}(w)$, is just the function obtained by conjugating the coefficients of $S$,
the following formula for $U(z,w)$ is a complex version of the formula (4.73) from \cite{Gar}:

$4U(z,w) = 2 \left[ A ( \bar{S}(w) ,w;z,w ) \bar{S}(w) w + A( z,S(z) ;z,w ) S ( z ) z \right]$

\noindent $+\int_{S(z)}^{w} \tfrac{n-2}{\bar{S}(t)-t} A ( \bar{S}(t) ,t;z,w ) \bar{S}(t) t+A ( \bar{S}(t) ,t;z,w ) \bar{S}(t) - A_t ( \bar{S}(t),t;z,w ) \bar{S}(t) t dt$

\noindent $-\int_{z}^{\bar{S}(w)} \tfrac{n-2}{s-S(s)} A(s,S(s);z,w) S(s) s$

\hspace{.8 in} $+A (s,S(s);z,w) S(s)-A_s(s,S(s);z,w)S(s)sds ,$

\noindent where we have also taken into account the Cauchy data in the above formula:
$$\left\{
\begin{array}{l}
U|_{\gamma} = z w / 2 \\
U_z|_{\gamma}= w/2 \\
U_w|_{\gamma}= z/2
\end{array}\right.
$$

Instead of using the Schwarz function, 
this formula can be written in terms of the conformal map $f$ from the disk by making the change of variables $t = f(\xi)$, $s = f(\eta)$,
and using the relation $S(f(\eta)) = \bar{f}(\frac{1}{\eta})$.
Making this change of variables, one can then consider the following problem whose solution would yield three-dimensional quadrature domains
with prescribed quadrature:

Suppose the singularities of $U(z,\bar{z})$ are prescribed throughout $\Omega$.
Is it possible to then determine $f$ using the formula above?

Conceptually, this is analogous to what was done in Section \ref{sec:QD}, recovering $f(z)$ from the functional equation (\ref{eq:fct1}),
but in practical terms it appears much more difficult.

{\bf 3.} Since $\RR^4$ has proven easier to work with than the more relevant $\RR^3$, 
one wonders if axially symmetric examples in $\RR^3$ can be expected to ``interpolate'', in some sense, between ones in $\RR^2$ and corresponding ones in $\RR^4$,
so that one can formulate conjectures regarding the more physically relevant case of $\RR^3$.
For instance, starting with a circle, a sphere in $\RR^3$, and a hypersphere in $\RR^4$ (all of the same radius), 
consider a dipole flow from the center (in the direction of the  $x$-axis).
The two and four-dimensional examples each develop a cusp on the $x$-axis.
We expect the three-dimensional example does also, 
and it is tempting to conjecture that the position of the cusp is located between the positions of the two and four-dimensional cases.
It may be more suitable to formulate this question generally within the framework of ``generalized axially symmetric potentials'' \cite{Wein53},
i.e. considering non-integer values of the parameter $\lambda$ in the equation $$\Delta u + \frac{\lambda u_x}{x} = 0.$$
Then the question is whether certain aspects of Laplacian growth ``depend monotonically'' on $\lambda$.

{\bf 4.} The procedure used in the proof of Theorem \ref{thm:nonalg} can also be used to construct other domains having quadrature formula 
of the form (\ref{eq:QF3}) with any choice of $m$, i.e. having any number of non-vanishing harmonic moments.
Considering these other cases, one observes that only exceptional cases are algebraic.
It is tempting to conjecture that quadrature domains are almost never algebraic in dimensions greater than two,
and perhaps this would help explain why it has been so difficult to obtain exact solutions to the higher-dimensional Laplacian growth problem.
For Laplacian growth in any number of dimensions, 
if the domain is initially a quadrature domain, then as mentioned in Section \ref{sec:LG} it remains a quadrature domain,
but in the plane it moreover remains an algebraic curve of \emph{bounded} (uniformly throughout time) degree.

{\bf 5.} Instead of prescribing the coefficients in the quadrature formula and trying to find the domain, one can consider a given domain, say with algebraic boundary, 
and try to determine if it admits a quadrature formula.
In this case, it is more reasonable to consider ``quadrature domains in the wide sense'' 
and allow for a measure supported on a continuum instead of a finite sum of point-evaluations.
For an analytic surface, it is always possible to find such a measure supported on a slightly smaller domain.
For ellipsoids, there is a measure supported on a two dimensional set inside called the ``focal ellipse'' \cite{KS89}.

\noindent {\bf Acknowledgement:} We wish to thank Lavi Karp for directing us to an important reference \cite{Sretenskii38}.

{\em Purdue University

West Lafayette, IN 47907

eremenko@math.purdue.edu

elundber@math.purdue.edu}

\end{document}